\documentclass[11pt]{amsart}

\usepackage{amsmath}
\usepackage{amssymb}
\usepackage{amsthm}
\usepackage{amscd, amsfonts, mathrsfs}
\usepackage[dvips]{epsfig}
\usepackage{verbatim}

\usepackage{epsf}
\newcommand{\Ga}{\Gamma}

\newcommand{\la}{\lambda}

\newcommand{\si}{\sigma}
\newcommand{\Si}{\Sigma}

\newcommand{\Om}{\Omega}

\newcommand{\cA}{\mathcal{A}}

\newcommand{\cM}{\mathcal{M}}
\newcommand{\cN}{\mathcal{N}}

\newcommand{\cS}{\mathcal{S}}
\newcommand{\RR}{\mathbb R}

\newcommand{\rar}{\rightarrow}

\newcommand{\dd}{\operatorname{div}}
\newcommand{\sq}{\sqrt{1 + \phi^2 |\nabla f|_g^2}}
\newcommand{\sqsq}{1 + \phi^2 |\nabla f|_g^2}

\theoremstyle{plain}
\newtheorem{theorem}{Theorem}[section]
\newtheorem{lemma}[theorem]{Lemma}

\theoremstyle{definition}
\newtheorem{rema}[theorem]{Remark}
\newtheorem{defi}[theorem]{Definition}

\numberwithin{equation}{section}
\title[On the Penrose Inequality for Charged Black Holes]{On the Penrose Inequality for Charged Black Holes}

\author[Disconzi]{Marcelo M. Disconzi}
\address{Department of Mathematics\\
Stony Brook University\\ Stony Brook, NY 11794}
\curraddr{Department of Mathematics\\
Vanderbilt University\\
Nashville, TN 37240}
\email{disconzi@math.sunysb.edu, marcelo.disconzi@vanderbilt.edu}

\author[Khuri]{Marcus A. Khuri}
\address{Department of Mathematics\\
Stony Brook University\\ Stony Brook, NY 11794}
\email{khuri@math.sunysb.edu}
\thanks{The second author is partially supported by
NSF Grant DMS-1007156 and a Sloan Research Fellowship.}
\begin{document}

\begin{abstract}
In \cite{BK1} and \cite{BK2}, Bray and Khuri outlined an approach to prove the Penrose inequality for general
initial data sets of the Einstein equations. In this paper we extend this approach so that it may be applied
to a charged version of the Penrose inequality. Moreover, assuming that the initial data is time symmetric,
we prove the rigidity statement in the case of equality for the charged Penrose inequality, a result which
seems to be absent from the literature. A new quasi-local mass, tailored to charged initial data sets is
also introduced, and used in the proof.
\end{abstract}

\maketitle

\tableofcontents

\section{Introduction}
Despite its tremendous success, the question of whether General Relativity
provides a consistent (classical) physical theory remains open\footnote{See \cite{Da} for
an overview which includes the experimental confirmations of the theory;
an account of recent advances in its mathematical aspects can be found in \cite{CGD}.}. Physical consistency, here, should
be understood as the statement that a large class of initial states, ``generic'' in some sense,
do not evolve into unphysical space-times. Given that the
singularity theorems \cite{HE} ensure that the time evolution of physically reasonable initial data sets
can be geodesically incomplete, one expects that such singularities cannot be observed from the asymptotic region if
the theory is capable of making sound predictions. This idea is usually referred to as the
(weak) Cosmic Censorship Conjecture, and it can be stated as follows \cite{C,GH}.

\emph{The maximal Cauchy development of generic asymptotically flat initial data for the Einstein
equations, with physically reasonable sources, possesses a complete $\mathscr{I}^+$.}

The solution of this conjecture turns out to be a very difficult problem, as it requires
understanding the full time development of the space-time. Hence it is pragmatic to investigate
a more simple and related issue, which can be addressed entirely from the point of view of a Cauchy surface.
More precisely, one seeks to establish properties of initial data that ought to
be satisfied if the conjecture is true. In this regard, heuristic arguments of
Penrose \cite{Pen} suggest that, for physically reasonable initial data sets,
the validity of weak Cosmic Censorship implies the following inequality
\begin{gather}
E_{ADM} \geq \sqrt{\frac{\cA}{16 \pi}},
\label{Penrose_simple}
\end{gather}
where $E_{ADM}$ is the ADM energy \cite{ADM} and $\cA$ is the
minimal area required to enclose the apparent horizon. This inequality is known as the Penrose inequality, and it
is usually accompanied by a rigidity statement, very much like the rigidity statement of the
positive mass theorem \cite{SY2,SY3, W}, saying that if equality happens then
the initial data embed into the Schwarzchild space-time.
The Penrose inequality, along with the rigidity statement, has been established in the time-symmetric case
by Huisken and Ilmanen \cite{HI} for one black hole, and independently by Bray \cite{Bray1} for
multiple black holes.

Inequality (\ref{Penrose_simple}) admits several
generalizations, such as the inclusion of charge, and the inclusion of angular
momentum. As these generalizations describe natural and relevant
physical scenarios, the validity of the Penrose inequality in such cases is also an important test for the
Cosmic Censorship conjecture. See \cite{Mar} for an account of the current status of this topic.

The present work is concerned with an initial data set $(\cM, g, k, E)$ for a single electrically charged black hole,
in which case the Penrose inequality reads
\begin{gather}
 E_{ADM} \geq \sqrt{ \frac{\cA }{16\pi} } + \sqrt{ \frac{\pi}{\cA } } e^2,
\label{Penrose_charged}
\end{gather}
where $e=\lim_{r\rightarrow\infty}\frac{1}{4\pi}\int_{S_{r}}E^{i}\nu_{i}$ is the total electric charge, with $S_{r}$ coordinate spheres in
the asymptotic end having unit outer normal $\nu$. Under appropriate energy conditions, inequality (\ref{Penrose_charged})
has been proven by Jang for time-symmetric
initial data under the assumption that a smooth solution to the Inverse Mean Curvature Flow (IMCF)
exists \cite{Ja2}. However in light of Huisken and Ilmanen's work \cite{HI}, the hypothesis of a smooth IMCF can be discarded.
A proof of this inequality in the non-time-symmetric, spherically symmetric case has also be given in \cite{Hay}.
To the best of our knowledge, the rigidity statement in the case of equality does not appear in the literature
and does not follow directly from Jang's original argument.
In this paper, we show that under the assumption of a solution to the coupled Jang-IMCF
system of equations \cite{BK1,BK2},
(\ref{Penrose_charged}) is valid in the non-time-symmetric case. We also establish the rigidity
statement both when $k=0$ and $k \neq 0$, and in the former case without the assumption of a solution
to the Jang-IMCF system. Related results for the positive mass theorem with charge are proven in \cite{KW}.

\begin{theorem}
Let $(\cM, g, k, E)$ be a $3$-dimensional, asymptotically flat initial data set for the
Einstein-Maxwell system with a
connected outermost apparent horizon boundary $\partial \cM$, where $E$ is the electric field.
Assume that the charge density is zero $\dd E = 0$, that the magnetic field vanishes, and that the non-electromagnetic matter fields
satisfy the dominant energy condition.
If the coupled Jang-IMCF system of equations admits a solution
with a weak IMCF (in the sense of \cite{HI}) such that the boundary of the Jang surface is minimal,
then (\ref{Penrose_charged}) holds and if equality is attained the initial data
arise from the Reissner-Nordstrom spacetime. Furthermore, if $k=0$, the same conclusions hold without
the assumption concerning solutions to the Jang-IMCF system.
\label{Penrose_inequality_theo}
\end{theorem}

The hypothesis of connectedness of $\partial \cM$ is necessary, as inequality
(\ref{Penrose_charged}) is known to be false when multiple black holes are present \cite{WY}.
Similarly, the inequality need not be true if electromagnetic sources are allowed
outside the black hole \cite{Hay}.

Theorem \ref{Penrose_inequality_theo} is established in Section \ref{proof}, with
the rigidity statement shown in Section \ref{case_equality}. Its proof
makes use of a new quasi-local mass which generalizes the Hawking mass for
charged initial data. This mass is introduced in Section \ref{charged_Hawking_mass_section},
where it is motivated from elementary principles. In Section \ref{gen_Jang_surface_section} we
recall, along with some important definitions, how to construct a generalized
Jang surface from the ideas of \cite{BK1,BK2}.

\section{Charged Jang Deformation\label{gen_Jang_surface_section}}

Let $(\cM, g, k, E)$ be a $3$-dimensional initial data set, consisting of a Riemannian metric $g$, a symmetric 2-tensor $k$,
and a vector field $E$ representing the electric field. It is assumed that the manifold has a single component boundary consisting of an
apparent horizon and that there are no other apparent horizons present. Moreover the data are taken to be asymptotically flat with one end, in that
outside a compact set the manifold is diffeomorphic to the complement of a ball in $\mathbb{R}^{3}$, and in the coordinates given
by this asymptotic diffeomorphism the following fall-off conditions hold
\begin{equation*}
|\partial^{m}(g_{ij}-\delta_{ij})|=O(|x|^{-m-1}),\text{ }\text{
}\text{ }|\partial^{m}k_{ij}|=O(|x|^{-m-2}),\text{ }\text{
}\text{ }|\partial^{m}E^{i}|=O(|x|^{-m-2}),\text{ }\text{ }\text{ }m=0,1,2,\text{ }\text{ }\text{as}\text{ }\text{
}|x|\rightarrow\infty.
\end{equation*}
With a vanishing magnetic field, the matter and current densities for the non-electromagnetic matter fields are given by\footnote{We use conventions where the right hand side of the Einstein
equations does not have $8\pi$, and we set $G_N = c =1$.}
\begin{align}
\begin{split}
 2 \mu  & = R + (Tr k)^2 - |k|_{g}^2 - 2|E|_{g}^2, \\
J & = \dd (k - (Tr k)g),
%\label{matter_current}
\nonumber
\end{split}
\end{align}
where $R$ denotes the scalar curvature of $g$. The following inequality will be referred to as the dominant energy condition
\begin{gather}
\mu \geq |J|_{g}.
\label{dec}
\end{gather}
In the time-symmetric case when $k=0$, this condition states that
\begin{equation}\label{1}
R\geq2|E|_{g}^{2},
\end{equation}
and is heavily relied upon in the proof of the
charged Penrose inequality. In fact the main difficulty in extending the proof to the non-time-symmetric case, is the lack of this
inequality under the dominant energy condition assumption. For this reason we seek a deformation of the initial data to a new set
$(\Sigma,\overline{g},\overline{E})$, where $\Sigma$ is diffeomorphic to $\mathcal{M}$, and the metric $\overline{g}$ and vector field
$\overline{E}$ are related to $g$ and $E$ in a precise way described below. The purpose of the deformation is to obtain new initial
data which satisfies \eqref{1} in a weak sense, while preserving all other quantities appearing in the charged Penrose inequality,
such as the charge density, total charge, ADM energy, and boundary area.

Consider the warped product 4-manifold $(\cM \times \RR, \, g + \phi^2 dt^2)$, where $\phi$ is a nonnegative function to be chosen appropriately. Let $\Sigma=\{t=f(x)\}$ be
the graph of a function $f$ inside this warped product setting, then the induced metric on $\Sigma$ is given by $\overline{g}=g+\phi^{2}df^{2}$. In \cite{BK1,BK2} it is
shown that in order to obtain the most desirable positivity property for the scalar curvature of the graph, the function $f$ should satisfy
\begin{gather}
 \left( g^{ij} - \frac{\phi^2 f^i  f^j}{1 + \phi^2 |\nabla f|_g^2 }\right)
\left( \frac{ \phi \nabla_{ij}f
+ \phi_i f_j +  \phi_j  f_i}{ \sqrt{1 + \phi^2 |\nabla f|_g^2 }}
-k_{ij} \right) = 0,
\label{generalized_Jang_equation}
\end{gather}
where $\nabla$ denotes covariant differentiation with respect to the metric $g$, $f_{i}=\partial_{i}f$,
and $f^{i}=g^{ij}f_{j}$. Equation (\ref{generalized_Jang_equation}) is referred to as the generalized Jang equation, and when it is
satisfied $\Sigma$ will be called the Jang surface. This equation is quasi-linear elliptic, and degenerates when either $\phi=0$
or $f$ blows-up. The existence, regularity, and blow-up behavior for the generalized Jang equation is studied at length in \cite{HK}.
The scalar curvature of the Jang surface \cite{BK1,BK2} is given by
\begin{equation}\label{5}
\overline{R}=2(\mu-J(w))+2|E|_{g}^{2}+
|h-k|_{\overline{g}}^{2}+2|q|_{\overline{g}}^{2}
-2\phi^{-1}\overline{\dd}(\phi q),
\end{equation}
here $\overline{\dd}$ is the divergence operator with respect to $\overline{g}$,
$h$ is the second fundamental form of the graph $t=f(x)$ in the Lorentzian 4-manifold $(\cM \times \RR, \, \overline{g}-\phi^2 dt^2)$,
and $w$ and $q$ are 1-forms given by
\begin{equation*}\label{6}
h_{ij}=\frac{ \phi \nabla_{ij}f
+ \phi_i f_j +  \phi_j  f_i}{ \sqrt{1 + \phi^2 |\nabla f|_g^2 }},\text{ }\text{ }\text{ }\text{ }
w_{i}=\frac{\phi f_{i}}{\sqrt{1+\phi^{2}|\nabla f|^{2}}},\text{
}\text{ }\text{ }\text{ }
q_{i}=\frac{\phi f^{j}}{\sqrt{1+\phi^{2}|\nabla f|^{2}}}(h_{ij}-k_{ij}).
\end{equation*}

If the dominant energy condition is satisfied,
then all terms appearing on the right-hand side of \eqref{5} are
nonnegative, except possibly the last term.
However the last term has a special structure, and in many applications
it is clear that a specific choice
of $\phi$ will allow one to `integrate away' this divergence term, so
that in effect the scalar curvature
is weakly nonnegative (that is, nonnegative when integrated against
certain functions). For the charged Penrose inequality, a stronger condition than simple nonnegativity is required, more precisely
we seek an inequality (holding in the weak sense) of the following form
\begin{equation}\label{7}
\overline{R}\geq 2|\overline{E}|_{\overline{g}}^{2},
\end{equation}
where $\overline{E}$ is an auxiliary electric field defined on the Jang surface. This auxiliary electric field is required to satisfy three properties,
namely
\begin{equation}\label{8}
|E|_{g}\geq|\overline{E}|_{\overline{g}},\text{ }\text{ }\text{ }\text{ }\overline{\dd} \, \overline{E}=0,\text{ }\text{ }\text{ }\text{ }\overline{e}=e,
\end{equation}
where $\overline{e}$ is the total charge defined with respect to $\overline{E}$. In particular, if the first inequality of \eqref{8} is satisfied, then the dominant
energy condition \eqref{dec} and the scalar curvature formula \eqref{5} imply that \eqref{7} holds weakly. It turns out that there is a very natural choice for
this auxiliary electric field, namely $\overline{E}$ is the induced electric field on the Jang surface $\Sigma$ arising from the field strength $F$ of the electromagnetic
field on $(\cM \times \RR, \, g + \phi^2 dt^2)$. More precisely $\overline{E}_i = F(N,X_i)$,
where $N$ and $X_i$ are respectively the unit normal and canonical tangent vectors to $\Si$
\begin{equation*}\label{9}
N=\frac{\phi^{-1}\partial_{t}-\phi f^{i}\partial_{i}}{\sqrt{1+\phi^{2}|\nabla f|_{g}^{2}}},\text{ }\text{ }\text{ }\text{ }X_{i}=\partial_{i}+f_{i}\partial_{t},
\end{equation*}
and $F = \frac{1}{2} F_{ab}dx^a\wedge dx^b$ is given by
$F_{0i} = \phi E_i$ and $F_{ij} = 0$ for $i=1,2,3$, with $x^i$, $i=1,2,3$ coordinates on $\cM$ and $x^{0}=t$. In matrix form
\begin{gather}
F = \left(
\begin{array}{cccc}
  0 & \phi E_1 & \phi E_2 & \phi E_3 \\
- \phi E_1 & 0 & 0 & 0 \\
- \phi E_2 & 0 & 0 & 0 \\
- \phi E_3 & 0 & 0 & 0
 \end{array}
\right).
\nonumber
\end{gather}
In the appendix it is shown that
\begin{gather}
\overline{E}_i = \frac{E_i + \phi^2 f_i f^j E_j}{\sqrt{1 + \phi^2 |\nabla f|^2_g}},
\label{definition_E_bar}
\end{gather}
and that all the desired properties of \eqref{8} hold.

\begin{defi}
When $f$ solves \eqref{generalized_Jang_equation} and $\overline{E}$ is given by \eqref{definition_E_bar}, the
triple $(\Sigma,\overline{g},\overline{E})$ is referred to as charged Jang initial data.
\end{defi}

In order to apply these constructions to the charged Penrose inequality, we need not only the (weak version of) inequality \eqref{7},
but also three other properties of the charged Jang initial data. Let $S_{0}\subset\mathcal{M}$ denote
the outermost minimal area enclosure of $\partial\mathcal{M}$ (according to \cite{BT,HI}, $S_{0}$ exists, is unique, and is $C^{1,1}$),
and let $\mathcal{S}_{0}$ be the vertical lift of $S_{0}$ to $\Sigma$. Then the desired three properties are
\begin{equation}\label{10}
\overline{E}_{ADM}=E_{ADM},\text{ }\text{ }\text{ }\text{ }\text{ }|\mathcal{S}_{0}|_{\overline{g}}=|S_{0}|_{g}=:\mathcal{A},
\text{ }\text{ }\text{ }\text{ }\text{ }\overline{H}_{\mathcal{S}_{0}}=0,
\end{equation}
where $\overline{E}_{ADM}$ is the ADM energy of the Jang metric $\overline{g}$, and $|\mathcal{S}_{0}|_{\overline{g}}$ and $\overline{H}_{\mathcal{S}_{0}}$
are the area and mean curvature of the surface $\mathcal{S}_{0}$, respectively. The first of these equalities is achieved by imposing zero
Dirichlet boundary conditions for $f$ at spatial infinity. More precisely, if
\begin{equation}\label{11}
\phi(x)=1+\frac{C}{|x|}+O\left(\frac{1}{|x|^{2}}\right)
\text{ }\text{
}\text{ as }\text{ }\text{ }|x|\rightarrow\infty
\end{equation}
for some constant $C$, then according to \cite{HK}
\begin{equation}\label{12}
|\nabla^{m}f|(x)=O(|x|^{-\frac{1}{2}-m})\text{ }\text{ }\text{
as }\text{ }\text{ }|x|\rightarrow\infty,\text{ }\text{ }\text{
}\text{ }m=0,1,2,
\end{equation}
which is enough to ensure that the two ADM energies agree. The second equality of \eqref{10} may be obtained by imposing zero Dirichlet boundary conditions
for the warping factor $\phi|_{S_{0}}=0$ at the surface $S_{0}$. Notice that this conclusion should hold whether $f$ blows-up or
does not blow-up at $S_{0}$, since when blow-up occurs the Jang surface asymptotically approaches a cylinder over the blow-up region. It is well-known that
the Jang surface can only blow-up on the portion of $S_{0}$ which coincides with the apparent horizon boundary.
Lastly, the third equality of \eqref{10} is considered to be an appropriate boundary condition for the solutions of the generalized Jang equation \eqref{generalized_Jang_equation}.
Typically, on the portion of $S_{0}$ which coincides with the apparent horizon boundary, this boundary condition forces the solution $f$ to blow-up as just described, however this is not always the case.

The question now arises as to how one should choose the warping factor $\phi$. Let $\{\cS_\tau \}_{\tau= 0}^{\infty}$ be an IMCF inside the Jang surface $(\Si,\, \bar{g} )$,
starting from the minimal surface $\cS_0$. We claim that $\mathcal{S}_{0}$ is outerminimizing in $\Sigma$. To see this, consider a surface $\mathcal{S}\subset\Sigma$ which
encloses $\mathcal{S}_{0}$, and let $S\subset\mathcal{M}$ be its vertical projection. Then since $\overline{g}\geq g$ as positive definite matrices and $S_{0}$ is a minimal
area enclosure of $\partial\mathcal{M}$
\begin{equation*}\label{14}
|\mathcal{S}|_{\overline{g}}\geq |S|_{g}\geq |S_{0}|_{g}=|\mathcal{S}_{0}|_{\overline{g}},
\end{equation*}
where \eqref{10} has also been used. This shows that it is appropriate to start the IMCF from $\mathcal{S}_{0}$. We then set
\begin{gather}
\phi = \sqrt{\frac{|\cS_\tau|_{\overline{g}}}{16\pi} } \, \overline{H}_{\tau},
\label{definition_phi}
\end{gather}
where $|\mathcal{S}_{\tau}|_{\overline{g}}$ is the area of $\mathcal{S}_{\tau}$ and $\overline{H}_{\tau}$ is its mean curvature. This coincides with one of the choices for $\phi$ in
\cite{BK1,BK2}, where it is pertinent to the study of the Penrose inequality for general initial data. With this definition of $\phi$, the generalized Jang equation \eqref{generalized_Jang_equation}
is coupled in a nontrivial way to the IMCF equations. In this paper we will assume that a solution, with appropriate regularity, exists for this coupled Jang-IMCF system of equations such
that \eqref{10}, \eqref{11}, and \eqref{12} are valid. The appropriate regularity should be akin to the weak solution of IMCF appearing in the work of Huisken and Ilmanen \cite{HI}, and which
is sufficient for the arguments of Section \ref{proof} to go through. We point out that a smooth solution to this system of equations exists in the case of spherical symmetry \cite{BK2},
and a further discussion of this system appears in \cite{BK1}.

\section{Charged Hawking Mass\label{charged_Hawking_mass_section}}

Let us start introducing a new quasi-local mass tailored to initial data for the Einstein-Maxwell system, and which will play an important role
in establishing Theorem \ref{Penrose_inequality_theo}. The idea to derive this mass is as follows. Consider a spherically symmetric line element
\begin{gather}
\gamma = ds^2 + r^2(s) d\si^2 = \frac{1}{r_{,s}^2} dr^2 + r^2 d\si^2,
\label{spherically_sym_g}
\end{gather}
where $r_{,s}=\frac{dr}{ds}$ and $d\sigma^{2}$ is the round metric on $\mathbb{S}^{2}$, and recall the metric on the
$t=0$ slice of the Reissner-Nordstrom spacetime
\begin{gather}
g_{RN} = \left( 1 - \frac{2M}{r} + \frac{e^2}{r^2} \right)^{-1} dr^2 + r^2 d\si^2.
\label{RN_slice}
\end{gather}
By attempting to transform \eqref{spherically_sym_g} into the form \eqref{RN_slice}, we find that
\begin{equation*}
\gamma = \frac{1}{r_{,s}^2} dr^2 + r^2 d\si^2 = \left( 1 - \frac{2M(s)}{r} + \frac{e^2}{r^2} \right)^{-1} dr^2 + r^2 d\si^2,
\end{equation*}
where
\begin{gather}
M(s) = \frac{1}{2}r\left( 1 + \frac{e^2}{r^2} - r_{,s}^2 \right).
%\label{mass_1}
\nonumber
\end{gather}
Since the mean curvature and area of coordinate spheres in the metric $\gamma$ are given by
\begin{gather}
H = \frac{2r_{,s}}{r},\text{ }\text{ }\text{ }\text{ } | \partial B_r | = 4\pi r^2  = \int_{\partial B_r} r^2 d\si,
\label{mean_curvature_rotational}
\end{gather}
we obtain
\begin{gather}
M(r) = \sqrt{\frac{|\partial B_r|}{16\pi} }
\left(1 + \frac{4\pi e^2}{|\partial B_r|} - \frac{1}{16\pi}\int_{\partial B_r} H^2 \right).
\nonumber
\end{gather}

This motivates the following definition.

\begin{defi}
Given initial data $(\mathcal{M},g,E)$ for the Einstein-Maxwell equations and a closed 2-surface $\cS \subset \cM$, the charged
Hawking mass is defined to be
\begin{gather}
M_{CH}(\cS) = \sqrt{\frac{|\cS|}{16\pi} } \left(1 + \frac{4\pi e^2}{|\cS|} - \frac{1}{16\pi}\int_\cS H^2 \right),
%\label{charged_Hawking_mass}
\nonumber
\end{gather}
where $H$ is the mean curvature of $\mathcal{S}$ and $e=\frac{1}{4\pi}\int_{\mathcal{S}}E^{i}\nu_{i}$, with $\nu$ the unit normal pointing
towards spatial infinity. If $\mathcal{S}$ bounds a volume, then $e$ is the total charge contained within $\mathcal{S}$.
\end{defi}

Notice that $M_{CH}$ reduces to the ordinary Hawking mass when $e=0$. It should also be pointed out that
this definition agrees in spherical symmetry with the mass introduced in \cite{Hay}, although the derivation here is different.

\section{Proof of Theorem \ref{Penrose_inequality_theo}\label{proof}}

In this section we will establish inequality (\ref{Penrose_charged}) under the assumptions of Theorem \ref{Penrose_inequality_theo}. For the time being,
let us assume that there is a smooth IMCF $\{\mathcal{S}_{\tau}\}_{\tau=0}^{\infty}$, inside a Jang surface $\Sigma$ satisfying properties \eqref{10}.
Standard properties of the IMCF \cite{HI} imply that
\begin{align}
\begin{split}
\frac{d}{dt} M_{CH}(\cS_\tau) =& -\frac{1}{2}\sqrt{\frac{\pi}{|\cS_\tau|_{\overline{g}} }} e^2 +
\sqrt{\frac{|\cS_\tau|_{\overline{g}}}{16\pi}} \left( \frac{1}{2} - \frac{1}{4} \chi(\cS_\tau) \right) \\
&+ \frac{1}{16\pi}\sqrt{\frac{|\cS_\tau|_{\overline{g}}}{16\pi}}
\int_{\cS_\tau} \left( 2 \frac{ |\nabla_{\cS_\tau} \overline{H}|^2 }{\overline{H}^2}
+ |\overline{A}|^2 -\frac{1}{2}\overline{H}^2 + \overline{R} \right) ,
\label{derivative_charged_Hawking_mass}
\end{split}
\end{align}
where $\overline{A}$ and $\overline{H}$  are, respectively, the second fundamental form and
mean curvature of $\cS_\tau$, and $\chi$ is the Euler characteristic.
Since
\begin{gather}
|\overline{A}|^2 -\frac{1}{2}\overline{H}^2 =
\frac{1}{2}(\la_1 - \la_2 )^2,
%\label{16}
\nonumber
\end{gather}
where $\la_i$, $i=1,2$, are the principal curvatures of $\cS_\tau$, this term is nonnegative. The same
holds for the second term on the right hand side of
(\ref{derivative_charged_Hawking_mass}) as $\chi(\cS_\tau) \leq 2$ \cite{HI}. Therefore
(\ref{derivative_charged_Hawking_mass}) combined with \eqref{5} gives
\begin{gather}
\frac{d}{dt} M_{CH}(\cS_\tau) \geq
-\frac{1}{2}\sqrt{\frac{\pi}{|\cS_\tau|_{\overline{g}} }} e^2 +
\frac{1}{16\pi} \sqrt{\frac{|\cS_\tau|_{\overline{g}}}{16\pi}} \int_{\cS_\tau}
\left( 2|E|_g^2 -  \frac{2}{\phi} \overline{\dd}(\phi q) \right)
%\label{inequality_dM_dt}
\nonumber
\end{gather}
where the dominant energy condition (\ref{dec}) and the fact that $|w|_g \leq 1$ have
been used. From \eqref{8} and Holder's inequality it follows that
\begin{gather}
\int_{\cS_\tau} |E|_g^2 \geq \int_{\cS_\tau} |\overline{E}|_{\overline{g}}^2
\geq \int_{\cS_\tau} \big \langle \overline{E}, \nu_{\overline{g}} \rangle^2
\geq
\frac{\left( \int_{\cS_\tau} \langle \overline{E}, \nu_{\overline{g}} \rangle \right)^2}{ |\cS_\tau|_{\overline{g}} },
\label{int_E_square_inequality}
\end{gather}
where $\nu_{\overline{g}}$ in the unit outer normal to $\mathcal{S}_{\tau}$.
Applying the divergence theorem on the region $\Omega\subset\Si$ between $\cS_\tau$ and spatial infinity, and using \eqref{8},
produces
\begin{gather}
\int_{\cS_\tau} \langle \overline{E}, \nu_{\overline{g}} \rangle = - \int_{\Om} \overline{\dd} \, \overline{E}+4\pi\overline{e} = 4\pi e.
%\label{charge_surface_boundary}
\nonumber
\end{gather}
Hence
\begin{gather}
\frac{d }{dt}M_{CH}(\cS_\tau) \geq
- \frac{1}{16\pi}\sqrt{\frac{|\cS_\tau|_{\overline{g}}}{16\pi}} \int_{\cS_\tau}
\frac{2}{\phi} \overline{\dd}(\phi q) .
\label{inequality_with_div}
\end{gather}

The next step will be to integrate the above inequality between zero and infinity. Observe that since
\begin{gather}
M_{CH}(\cS_{\tau}) = \sqrt{\frac{\pi}{|\cS_{\tau}|_{\overline{g}}}}e^2 + M_H(\cS_{\tau})
\nonumber
\end{gather}
where $M_{H}$ denotes the unaltered Hawking mass, and $|\cS_{\tau}|_{\overline{g}}$ grows exponentially in $\tau$, we have that
\begin{gather}
\lim_{\tau \rar \infty} M_{CH}(\cS_\tau) = \overline{E}_{ADM} = E_{ADM}.
\nonumber
\end{gather}
On the other hand, since (by \eqref{10}) $\mathcal{S}_{0}$ is a minimal surface and $|\mathcal{S}_{0}|_{\overline{g}}=\mathcal{A}$,
\begin{gather}
M_{CH}(\cS_0) = \sqrt{ \frac{\mathcal{A} }{16\pi} }
\left(1 + \frac{4\pi e^2}{ \mathcal{A} } \right) .
\nonumber
\end{gather}
Therefore integrating (\ref{inequality_with_div}) yields
\begin{gather}
E_{ADM} -
\sqrt{ \frac{\cA }{16\pi} }
\left(1 + \frac{4\pi e^2}{ \cA  } \right)
\geq
- \frac{2}{(16\pi)^\frac{3}{2}} \int_{\Si} \frac{  \sqrt{|\cS_\tau|_{\overline{g}}} \, \overline{H} }{\phi} \overline{\dd}(\phi q) ,
\nonumber
\end{gather}
after applying the co-area formula. According to the definition \eqref{definition_phi} of $\phi$
\begin{gather}
\frac{1}{\sqrt{16\pi}} \int_{\Si} \frac{ \sqrt{|\cS_\tau|_{\overline{g}}} \,\overline{H} }{\phi} \overline{\dd}(\phi q)
=\int_{\Si}  \overline{\dd}(\phi q)
=\int_{\cS_\infty} \phi \langle q, \nu_{\overline{g}} \rangle
-\int_{\cS_0} \phi \langle q, \nu_{\overline{g}} \rangle.
\nonumber
\end{gather}
Well-known behavior of solutions to the (generalized) Jang equation \cite{HK,SY3} shows that $q(x) \rar 0$ as $|x|\rightarrow\infty$, and that $q$ remains bounded on $\mathcal{S}_{0}$
even if the Jang surface blows-up over part or all of this surface.
Moreover $\phi \rar 1$ as $|x|\rightarrow\infty$ and $\phi = 0$ on $\mathcal{S}_{0}$, since
$\overline{H}_{\mathcal{S}_{0}} = 0$ by \eqref{10}. Hence both boundary integrals vanish, and this yields the inequality (\ref{Penrose_charged}).

\begin{rema}
In this section it was assumed that the IMCF is smooth inside the Jang surface $\Sigma$. However it is only expected that a weak
solution of the Jang-IMCF system of equations will exist in general. If such a solution produces a weak IMCF in the sense of
Huisken and Ilmanen \cite{HI}, then it is easily verified that the arguments presented here will carry over to this setting.
In particular, $M_{CH}$ is monotonic across jumps in the weak IMCF, so that inequality (\ref{Penrose_charged}) is established
rigorously for time-symmetric initial data when $k=0$, as it is not necessary to solve the generalized Jang equation in this case.
\end{rema}

\section{Case of Equality\label{case_equality}}

In this section we will show that if equality holds in (\ref{Penrose_charged}), then the given initial data $(\mathcal{M},g,k,E)$ arise
from the exterior region of the Reissner-Nordstrom spacetime having metric
\begin{equation*}\label{15.0}
-\phi_{RN}^{2}dt^{2}+g_{RN}\text{ }\text{ }\text{ for }\text{ }\text{ }r\geq M+\sqrt{M^{2}-e^{2}},\text{ }\text{ }\text{ }\text{ }M\geq e,
\end{equation*}
where
\begin{equation*}\label{15.1}
\phi_{RN}=\sqrt{1-\frac{2M}{r}+\frac{e^2}{r^{2}}},\text{ }\text{ }\text{ }\text{ }g_{RN}=\left(1-\frac{2M}{r}+\frac{e^2}{r^{2}}\right)^{-1}dr^{2}+r^{2}d\sigma^{2}.
\end{equation*}
As mentioned in the introduction, it seems that this result does not appear in the literature,
and furthermore it does not follow from Jang's original arguments \cite{Ja} even in the time-symmetric case.

Suppose that equality holds in \eqref{Penrose_charged}, then all inequalities appearing in Section \ref{proof} must be equalities and
the following quantities must vanish
\begin{equation}\label{15}
|\nabla_{\mathcal{S}_{\tau}}\overline{H}|=|\overline{A}|-\frac{1}{2}\overline{H}^{2}=\mu-J(w)=|h-k|_{\overline{g}}=|q|_{\overline{g}}\equiv 0.
\end{equation}
In fact
\begin{equation}\label{17}
\mu=|J|_{g}\equiv 0,
\end{equation}
as can be seen from the identity
\begin{equation*}
\mu-J(w)=(\mu-|J|_{g})+|J|_{g}(1-|w|_{g})+(|J|_{g}|w|_{g}-J(w)),
\end{equation*}
combined with the dominant energy condition \eqref{dec} and the inequality $|w|_{g}<1$, which is valid wherever the solution $f$ to the generalized
Jang equation does not blow-up, that is, away from $\partial\mathcal{M}$. Moreover \eqref{17} and \eqref{15} imply that on each surface $\mathcal{S}_{\tau}$
\begin{equation}\label{18}
\overline{H}=const,\text{ }\text{ }\text{ }\text{ }|\overline{A}|=const,\text{ }\text{ }\text{ }\text{ }\lambda_{1}=\lambda_{2}=const.
\end{equation}

\subsection{Time-Symmetric Case\label{k_zero}}

In this subsection we restrict attention to the case when $k=0$. Here the solution of the generalized Jang equation is $f=0$, and
the IMCF, which is guaranteed to be smooth \cite{HI}, starts at the outermost minimal surface boundary $\mathcal{S}_{0}=\partial\mathcal{M}$.
We will show that the initial data $(\mathcal{M},g,E)$ are equivalent to that of the exterior region of the $t=0$ slice of the Reissner-Nordstrom spacetime.

When $k=0$, \eqref{17} implies that
\begin{gather}
R = 2|E|_g^2.
%\label{R_equal}
\nonumber
\end{gather}
Furthermore equality in \eqref{int_E_square_inequality} yields
\begin{equation}\label{19}
\langle E,\nu_{g}\rangle=const\text{ }\text{ }\text{ and }\text{ }\text{ }E=a(\tau)\nu_{g}\text{ }\text{ }\text{ on }\text{ }\text{ }\mathcal{S}_{\tau},
\end{equation}
for some smooth function $a(\tau)$. It follows that the scalar curvature $R$ is constant on $\mathcal{S}_{\tau}$.
Next observe that under IMCF $\partial_\tau = H^{-1} \nu_{g}$, so that taking a trace of the Riccati equation produces
\begin{gather}
 \partial_\tau H = - \Delta_{\cS_\tau} H^{-1} - (Ric(\nu,\nu) + |A|^2 ) H^{-1}
\nonumber
\end{gather}
where $Ric(\nu,\nu)$ is Ricci curvature in the direction $\nu_{g}$.
It then follows from \eqref{18} that $Ric(\nu,\nu)=const$ on $\mathcal{S}_{\tau}$.
Thus in light of two traces of the Gauss equation
\begin{gather}
Ric(\nu,\nu) - \frac{1}{2} R = - K + \frac{1}{2}H^2 - \frac{1}{2} |A|^2,
\nonumber
\end{gather}
the Gauss curvature satisfies $K=const$ on $\mathcal{S}_{\tau}$. It follows that $\cS_\tau$ is isometric to a round sphere,
having induced metric $r^{2}(\tau)d\sigma^{2}$ for some smooth function $r(\tau)$.
By the Gauss lemma
\begin{gather}
g = H^{-2} d\tau^2 + g|_{\cS_\tau} = \frac{4H^{-2}}{r^2} dr^2 + r^2 d\si^2,
\label{g_Gauss}
\end{gather}
where we have also used the fact that $d\tau = 2r^{-1}dr$, which follows from $4 \pi r^2(\tau) = |\cS_\tau|_g = |\cS_0|_g e^\tau$.
Since $\frac{d}{d\tau}M_{CH}\geq 0$, we must have $\frac{d}{d\tau}M_{CH}=0$ so that $M_{CH}(\tau)=M$ is constant for all $\tau$.
We may then solve for $H$ in terms of $M$ to find
\begin{gather}
H^2 = \frac{4}{r^2} \left( 1 - \frac{2M}{r} + \frac{e^2}{r^2} \right).
\nonumber
\end{gather}
Combining this with (\ref{g_Gauss}) shows that $g$ agrees with the metric $g_{RN}$ on the $t=0$ slice
of the Reissner-Nordstrom spacetime.

In order to complete the proof, we must show that the given electric field $E$ agrees
with the electric field arising from a point charge at the origin with respect to the metric $g_{RN}$:
\begin{equation*}\label{E_correct_form}
E_{RN}= \frac{e}{r^2} \sqrt{1 - \frac{2M}{r} + \frac{e^2}{r^2}} \, \partial_r.
\end{equation*}
To see that this is in fact the case, apply \eqref{19} to find
\begin{gather}
4\pi e = \int_{S_r} \langle E,\, \nu_{g} \rangle = a(r) \int_{S_r} r^2 \, d\si = 4\pi r^2 a(r).
\nonumber
\end{gather}
It follows that $E = \frac{e}{r^2} \nu_{g}$. Since
\begin{equation*}
\nu_{g} = \sqrt{1 - \frac{2M}{r} + \frac{e^2}{r^2}} \, \partial_r,
\end{equation*}
the desired result follows.

\subsection{Non-Time-Symmetric Case\label{case_k_not_zero}}
Suppose now that $k$ does not vanish. We may apply the arguments from the previous section to the charged Jang initial data
$(\Sigma,\overline{g},\overline{E})$ to obtain
\begin{equation*}\label{20}
\overline{g}=g_{RN},\text{ }\text{ }\text{ }\text{ }\text{ }\overline{E}=E_{RN}.
\end{equation*}
Furthermore since $\overline{g}$ is spherically symmetric we can use
(\ref{mean_curvature_rotational}) to calculate
\begin{gather}
\overline{H} = \frac{2}{r} \sqrt{ 1- \frac{2M}{r} + \frac{e^2}{r^2} },
\nonumber
\end{gather}
from which it follows that
\begin{align}
\phi = \sqrt{\frac{|\cS_\tau|_{\overline{g}}}{16\pi}} \, \overline{H}
= \frac{1}{2}r \, \overline{H}
=\sqrt{ 1- \frac{2M}{r} + \frac{e^2}{r^2} },
\nonumber
\end{align}
and hence $\phi=\phi_{RN}$. Since $g=\overline{g}-\phi^{2}df^{2}$ we find that the map $I:x\mapsto(x,f(x))$ yields an isometric embedding
of $(\mathcal{M},g)$ into the Reissner-Nordstrom spacetime. Notice that this implies that $S_{0}$ (defined in Section \ref{gen_Jang_surface_section})
is an apparent horizon, however by assumption the only apparent horizon in $\mathcal{M}$ is the boundary, and so $S_{0}=\partial\mathcal{M}$. That is,
the boundary of $\mathcal{M}$ is outerminimizing in the case of equality. Next recall that since $|h-k|_{\overline{g}}=0$, a calculation \cite{BK1,BK2} shows
that the second fundamental form of the embedding $I$ agrees with the initial data $k$.

It remains to show that $E$ agrees with the induced electric field on the surface $I(\mathcal{M})$ sitting inside the Reissner-Nordstrom
spacetime. Consider the field strength tensor for the electromagnetic field in the Reissner-Nordstrom spacetime
\begin{gather}
F_{RN} =  \left(
 \begin{array}{cccc}
  0 & \phi_{RN}(E_{RN})_1 & \phi_{RN}(E_{RN})_2 & \phi_{RN}(E_{RN})_3 \\
- \phi_{RN}(E_{RN})_1 & 0 & 0 & 0 \\
- \phi_{RN}(E_{RN})_2 & 0 & 0 & 0 \\
- \phi_{RN}(E_{RN})_3 & 0 & 0 & 0
 \end{array}
\right)
=\left(
 \begin{array}{cccc}
  0 & \frac{e}{r^{2}} & 0 & 0 \\
- \frac{e}{r^{2}} & 0 & 0 & 0 \\
0 & 0 & 0 & 0 \\
0 & 0 & 0 & 0
 \end{array}
\right).
\nonumber
\end{gather}
The induced electric field on the surface $I(\mathcal{M})$ is given by
\begin{gather}
\widehat{E}_i = F_{RN}(\widehat{N}, X_i),
\nonumber
\end{gather}
where
\begin{equation*}
\widehat{N}=\frac{\phi_{RN}^{-1}\partial_{t}+\phi_{RN}g_{RN}^{il}f_{l}\partial_{i}}{\sqrt{1-\phi_{RN}^{2}|\nabla f|^{2}_{g_{RN}}}}
\end{equation*}
is the unit normal to $I(\mathcal{M})$, and $X_i = \partial_i + f_i \partial_t$, $i=1,2,3$ are tangent vectors.
It is shown in Appendix \ref{Lorentzian_case} that
\begin{gather}
E_{i}=\widehat{E}_i,
\label{hat_E_i_equal_E_i}
\end{gather}
finishing the proof.

\appendix\section*{Appendix: Properties of $E$ and $\overline{E}$}

The induced electric field on a spacelike slice $\Sigma$ of a spacetime $(\cN,\eta)$ is given by $F(N, \cdot)$, where
$F$ is the field strength of the electromagnetic field and $N$ is the unit normal to $\Sigma$.
Although the physical situation occurs when $\eta$ has Lorentzian signature,
the same procedure can be applied in the Riemannian setting and is utilized in this paper.
The point of interest here is of course when $\Sigma$ is the Jang surface. Both cases of an ambient spacetime
having Riemannian and Lorentzian signature are analyzed; the former is treated in Appendix \ref{Riemannian_case},
while the later is treated in Appendix \ref{Lorentzian_case}. In particular we prove \eqref{8}, \eqref{definition_E_bar},
and (\ref{hat_E_i_equal_E_i}).

\section{Riemannian case\label{Riemannian_case}}

Throughout this section $\Si = \{t=f(x)\}$ will be a graph inside
the warped product 4-manifold $(\cM \times \RR, \, g + \phi^2 dt^2)$. Although we will use the same notation as in
the body of the paper, the function $f$ will not be required to satisfy the generalized Jang equation (\ref{generalized_Jang_equation}).
More precisely, the validity of the results presented here is independent of any equation that $f$ satisfies.

Let $F = \frac{1}{2} F_{ab}dx^a\wedge dx^b$ be the field strength given by
$F_{0i} = \phi E_i$ and $F_{ij} = 0$ for $i,j=1,2,3$, with $x^i$ coordinates on $\cM$. In matrix form
\begin{gather}
F = \left(
 \begin{array}{cccc}
  0 & \phi E_1 & \phi E_2 & \phi E_3 \\
- \phi E_1 & 0 & 0 & 0 \\
- \phi E_2 & 0 & 0 & 0 \\
- \phi E_3 & 0 & 0 & 0
 \end{array}
\right).
\nonumber
\end{gather}
In order to check that this is the correct expression for $F$ in $(\cM \times \RR, \, g + \phi^2 dt^2)$, notice that
the unit normal to the $t=0$ slice is $\phi^{-1} \partial_t$, so that the electric field induced
on $\cM$ by $F(\phi^{-1} \partial_t, \cdot )$ is exactly $E$. Next observe that each tangent space of $ \Si$ is spanned
by $X_i = \partial_i + f_i \partial_t$, $i=1,2,3$, and its unit normal is given by
\begin{gather}
N = \frac{\phi^{-1}\partial_{t} - \phi f^i \partial_{i}}{\sqrt{1 +\phi^{2}|\nabla f|_g^{2}}} .
\nonumber
\end{gather}
Therefore the electric field induced on $\Sigma$ becomes
\begin{equation}
\overline{E}_i=F(N,X_i) =F\left(\frac{\phi^{-1}\partial_{t} - \phi f^l \partial_{l}}{\sqrt{1 +\phi^{2}|\nabla f|_g^{2}}}, \,
\partial_{i}+ f_i \partial_{t}\right)
=\frac{E_{i}+\phi^{2}f_i f^l E_{l}}{ \sqrt{1+\phi^{2}|\nabla f|_g^{2}} },
\label{E_bar_Riemmanian}
\end{equation}
which establishes (\ref{definition_E_bar}).

\begin{lemma}
\begin{equation*}\label{30}
|E|_{g}\geq|\overline{E}|_{\overline{g}}
\end{equation*}
\end{lemma}

\begin{proof}
By direct calculation
\begin{eqnarray*}
|\overline{E}|^{2}_{\overline{g}}&=&\overline{g}^{ij}F(N,X_{i})F(N,X_{j})\\
&=&\overline{g}^{ij}(1+\phi^{2}|\nabla f|_g^{2})^{-1}(E_{i}+\phi^{2}f_i f^l E_{l})(E_{j}+\phi^{2}f_j f^k E_{k})\\
&=&\frac{\overline{g}^{ij}E_{i}E_{j}}{1+\phi^{2}|\nabla f|_g^{2}}
+2\frac{\phi^{2}\overline{g}^{ij} f_i E_j f^k E_{k}}{(1+\phi^{2}|\nabla f|_g^{2})^{2}}
+\frac{\phi^{4}\overline{g}^{ij}  f_i f_j(f^k E_{k})^{2}}{(1+\phi^{2}|\nabla f|_g^{2})^{2}}.
\end{eqnarray*}
Observe that since
\begin{equation}
\overline{g}^{ij}=g^{ij}-\frac{\phi^{2}f^i f^j}{1+\phi^{2}|\nabla f|_g^{2}},
\label{induced_g_inverse}
\end{equation}
we have
\begin{eqnarray*}
& &\frac{\overline{g}^{ij}E_{i}E_{j}}{1+\phi^{2}|\nabla f|_g^{2}}+
2\frac{\phi^{2}\overline{g}^{ij}  f_i E_j f^k E_{k}}{(1+\phi^{2}|\nabla f|_g^{2})^{2}}
+\frac{\phi^{4}\overline{g}^{ij} f_i  f_j(f^k E_{k})^{2}}{(1+\phi^{2}|\nabla f|_g^{2})^{2}} \\
&=&\frac{|E|_g^{2}}{1+\phi^{2}|\nabla f|_g^{2}}+\frac{\phi^{2}(f^i E_{i})^{2}}{1+\phi^{2}|\nabla f|_g^{2}}.
\end{eqnarray*}
Hence
\begin{eqnarray*}
|E|_g^{2}-|\overline{E}|_{\overline{g}}^{2}&=&\phi^{2}(|\nabla f|_g^{2} \, |\overline{E}|_{\overline{g}}^{2}-( f^i E_{i})^{2})\\
&\geq&\phi^{2}(|\nabla f|_g^{2}\, |\overline{E}|_{\overline{g}}^{2}-|\nabla f|_g^{2} \, |E|_g^{2})\\
&=&-\phi^{2}|\nabla f|_g^{2}(|E|_g^{2}-|\overline{E}|_{\overline{g}}^{2}),
\end{eqnarray*}
from which the desired inequality follows.
\end{proof}

\begin{lemma}
\begin{equation}\label{31}
\dd E=\sqrt{1+\phi^{2}|\nabla f|_g^{2}}~\overline{\mathrm{div}}\,\overline{E}
\end{equation}
\end{lemma}

\begin{proof}
We will present two different proofs of this equality, one
emphasizing conceptual aspects and another one based on direct computation,
as we believe that both approaches may be useful in
further generalizations of Theorem \ref{Penrose_inequality_theo}.

\textit{First proof.} Consider the 4-current $\mathcal{J}_{b}=\widehat{\nabla}^{a}F_{ab}$, where $\widehat{\nabla}$ is the Levi-Civita connection associated with
the 4-metric $\widehat{g}=g+\phi^{2}dt^{2}$. The Christoffel symbols of this metric are given by
\begin{equation*}\label{32}
\widehat{\Gamma}_{00}^{0}=\widehat{\Gamma}_{ij}^{0}=\widehat{\Gamma}_{i0}^{j}=0,\text{ }\text{ }\text{ }\text{ }1\leq i,j\leq 3,
\end{equation*}
\begin{equation*}\label{33}
\widehat{\Gamma}_{i0}^{0}=\partial_{i}(\log\phi),\text{ }\text{ }\text{ }\text{ }\widehat{\Gamma}_{00}^{i}=-\phi\phi^{i},
\end{equation*}
and $\widehat{\Gamma}_{ij}^{l}$ agree with the Christoffel symbols $\Gamma_{ij}^{l}$ of $g$ when $1\leq i,j,l\leq 3$.
In what follows, indices $a,b,c$ will run from 0 to 3, while indices $i,j,l$ will run from 1 to 3. We then have
\begin{equation}\label{34}
\mathcal{J}_i = \mathcal{J}(\partial_{i}) = \widehat{\nabla}^{a} F_{ai} = 0.
\end{equation}
To see this observe that
\begin{align}\label{35}
\begin{split}
\widehat{\nabla}^0 F_{0i} &= \widehat{g}^{0b}\widehat{\nabla}_{b}F_{0i}\\
&=\widehat{g}^{0b}(\partial_b F_{0i} - \widehat{\Ga}_{b0}^c F_{ci}-\widehat{\Gamma}_{bi}^{c}F_{0c} )\\
&=-\phi^{-2}(\widehat{\Gamma}_{00}^{c}F_{ci}+\Gamma_{0i}^{j}F_{0j})=0,
\nonumber
\end{split}
\end{align}
and
\begin{align}
\begin{split}
\widehat{\nabla}^j F_{ji} &= \widehat{g}^{jb}\widehat{\nabla}_{b}F_{ji}\\
&=g^{jl}(\partial_l F_{ji} - \widehat{\Ga}_{lj}^{c} F_{ci}-\widehat{\Gamma}_{li}^{c}F_{jc} )\\
&=-g^{jl}(\widehat{\Gamma}_{lj}^{0}F_{0i}+\Gamma_{li}^{0}F_{j0})=0,
\nonumber
\end{split}
\end{align}
where we have used the fact that $F_{ab}$ does not depend on $t$ and $F_{ij}=0$ for $1\leq i,j\leq 3$.
Since $\mathrm{div}E$ represents the charge density on $\mathcal{M}$,
\begin{equation*}
\mathrm{div}E=-\mathcal{J}(\phi^{-1}\partial_{t}).
\end{equation*}
Now rewrite the unit normal to $\mathcal{M}$ in terms of the unit normal to $\Sigma$ by
\begin{equation*}
\phi^{-1}\partial_{t}=\sqrt{1+\phi^{2}|\nabla f|_g^{2}}\,N +\phi f^{i}\partial_{i}.
\end{equation*}
In light of \eqref{34}, it follows that
\begin{equation*}
\mathrm{div} E= -\mathcal{J}(\phi^{-1}\partial_{t}) =
-\sqrt{1+\phi^{2}|\nabla f|_g^{2}}\,\mathcal{J}(N) -\mathcal{J}(\phi f^{i}\partial_{i})
= -\sqrt{1+\phi^{2}|\nabla f|_g^{2}}\,\mathcal{J}(N).
\end{equation*}
On the other hand $\overline{\mathrm{div}}\,\overline{E}$ represents the charge density on $\Sigma$, so that
\begin{equation}\label{36}
\overline{\mathrm{div}}\,\overline{E}=-\mathcal{J}(N),
\end{equation}
and equality \eqref{31} follows.

Note that \eqref{36} can also be verified explicitly. Take Fermi coordinates on $\Sigma$, and observe that
\begin{equation*}
\mathcal{J}(N) = N^b \widehat{\nabla}^{a}F_{ab} = - \widehat{\nabla}^a N^b F_{ab} + \widehat{\nabla}^a (F_{aN}).
\end{equation*}
Since $\widehat{\nabla}^a N^b$ is symmetric and $F_{ab}$ is anti-symmetric in $a$ and $b$, the first term on the right-hand
side vanishes. Consider now the remaining term on the right-hand side. Since $\widehat{\Ga}^a_{NN} = 0$ it follows that
\begin{equation*}
\widehat{\nabla}^{N}(F_{NN})=\widehat{g}^{Na}\widehat{\nabla}_{a}(F_{NN})=\widehat{g}^{Na}(N(F_{NN})-\widehat{\Gamma}_{aN}^{c}F_{cN})=0.
\end{equation*}
Also
\begin{equation*}
\widehat{\nabla}^{j}(F_{jN})=\widehat{g}^{jl}\widehat{\nabla}_{l}(F_{jN})=\widehat{g}^{jl}(\partial_{l}F_{jN}-\widehat{\Gamma}_{lj}^{i}F_{iN})
=-\overline{g}^{jl}(\partial_{l}\overline{E}_{j}-\overline{\Gamma}_{lj}^{i}\overline{E}_{i})=-\overline{\mathrm{div}}\,\overline{E},
\end{equation*}
where $\overline{\Gamma}_{lj}^{i}$ are Christoffel symbols for the metric $\overline{g}$ induced on $\Sigma$.

\textit{Second proof.} We now verify \eqref{31} by direct computation. According to \cite{BK1, BK2}, the relationship between the Christoffel symbols
of $\overline{g}$ and $g$ is given by
\begin{equation*}
\overline{\Ga}^m_{ij} = \Ga_{ij}^{m} - \phi \phi^m f_i f_j + \frac{\phi p_{ij} f^m}{\sqrt{1+\phi^2|\nabla f|_g^2 }},
\end{equation*}
where
\begin{equation}
p_{ij} = \frac{ \phi \nabla_{ij}f + \phi_i f_j + \phi_j f_i + \phi^2 \phi^l f_l f_i f_j}{\sqrt{1 + \phi^2 |\nabla f|_g^2} }
\label{sff_graph}
\end{equation}
is the second fundamental form of $\Si$ inside $(\cM\times \RR, g + \phi^2 dt^2)$. Let $\overline{\nabla}$ denote covariant differentiation with respect to $\overline{g}$, then
\begin{align}
 \begin{split}
  \overline{\nabla}_j \overline{E}_i  = &
\partial_j \left( \frac{E_i + \phi^2 f^l E_l f_i }{\sq} \right) -
\overline{\Ga}_{ij}^m \frac{E_m + \phi^2 f^l E_l f_m }{\sq} \\
 = & \frac{1}{\sq} \partial_j E_i + \partial_j \left(\frac{1}{\sq}\right) E_i +
\partial_j \left( \frac{\phi^2 f^l E_l f_i }{\sq} \right) \\
& - \Ga_{ij}^m \frac{E_m + \phi^2 f^l E_l f_m }{\sq}
+ \left( \phi \phi^m f_i f_j - \frac{f^m p_{ij} \phi }{\sq} \right )\left(  \frac{E_m + \phi^2 f^l E_l f_m }{\sq} \right) .
\end{split}
\nonumber
\end{align}
The terms involving $\partial_j E_i$ and $\Ga_{ij}^m E_m$, as well as
$\partial_j \left( \frac{\phi^2 f^l E_l f_i }{\sq} \right)$ and
$\Ga_{ij}^m \left( \frac{\phi^2 f^l E_l f_m }{\sq} \right)$ combine to form covariant
derivatives, so that
\begin{align}
 \begin{split}
  \overline{\nabla}_j \overline{E}_i  = &
\frac{1}{\sq} \nabla_j E_i + \nabla_j \left( \frac{\phi^2 f^l E_l f_i }{\sq}\right)
+ \partial_j \left(\frac{1}{\sq}\right) E_i \\
& + \left( \phi \phi^m f_i f_j - \frac{f^m p_{ij} \phi }{\sq} \right )\left(  \frac{E_m + \phi^2 f^l E_l f_m }{\sq} \right) .
\end{split}
\nonumber
\end{align}
Set $p=g^{ij}p_{ij}$, take a trace with $\overline{g}^{ij}$, and use (\ref{induced_g_inverse}) to find
\begin{align}
\begin{split}
 \overline{\dd} \, \overline{E} & =
\frac{1}{\sq} \nabla^i E_i +
\nabla^i \left( \frac{\phi^2 f^l E_l f_i }{ \sq} \right)
 + g^{ij} \partial_j \left(\frac{1}{\sq}\right) E_i \\
& + \left( \phi |\nabla f|_g^2 \phi^m  - \frac{ p \phi f^m }{\sq} \right )\left(  \frac{E_m + \phi^2 f^l E_l f_m }{\sq} \right)
- \frac{\phi^2 f^i f^j}{(\sqsq)^\frac{3}{2} } \nabla_j E_i \\
& - \frac{\phi^2 f^i f^j }{\sqsq} \nabla_j \left(\frac{\phi^2 f^l E_l f_i }{\sq} \right )
- \frac{\phi^2 f^i f^j }{\sqsq} \partial_j \left( \frac{1}{\sq} \right) E_i \\
& - \frac{\phi^2 f^i f^j }{\sqsq} \left( \phi f_i f_j \phi^m - \frac{\phi p_{ij} f^m }{\sq} \right)
 \left(  \frac{E_m + \phi^2 f^l E_l f_m }{\sq} \right).
\end{split}
\label{big_identity_begin}
\end{align}

Next expand \eqref{big_identity_begin}
\begin{align}
\begin{split}
 \overline{\dd} \, \overline{E} & =
\frac{1}{\sq} \dd E +
\nabla^i \left( \frac{\phi^2 f^l E_l }{ \sq} \right) f_i +
\frac{\phi^2 f^l E_l }{\sq} \nabla^i f_i \\
& + g^{ij} \partial_j \left(\frac{1}{\sq}\right) E_i
 + \frac{\phi |\nabla f|_g^2 \phi^m E_m}{\sq}  + \frac{ \phi^3 |\nabla f|_g^2 f^l E_l \phi^m f_m }{\sq}
- \frac{p\phi  f^m E_m }{\sqsq}\\
&  - \frac{p\phi^3  f^l E_l |\nabla f|^2_g }{\sqsq} - \frac{\phi^2 f^i f^j \nabla_j E_i }{(\sqsq)^\frac{3}{2}}
- \frac{\phi^2 f^i f^j}{\sqsq} \nabla_j \left( \frac{\phi^2 f^l E_l }{\sq} \right) f_i \\
&- \frac{ \phi^4 f^i f^j f^l E_l}{(\sqsq)^\frac{3}{2} } \nabla_j f_i
 -\frac{\phi^2 f^i f^j}{\sqsq} \partial_j \left( \frac{1}{\sq} \right) E_i
- \frac{\phi^3 f^i f^j \phi^m E_m f_i f_j }{ (\sqsq)^\frac{3}{2}  } \\
& - \frac{\phi^5 f^i f^j f^l E_l \phi^m f_m f_i f_j }{(\sqsq)^\frac{3}{2} }
+ \frac{\phi^3 f^i f^j p_{ij} f^m E_m }{(\sqsq)^2}
+ \frac{\phi^5 f^i f^j p_{ij} f^m E_m |\nabla f|^2_g }{(\sqsq)^2}.
\end{split}
\label{big_identity}
\end{align}
With the help of (\ref{sff_graph}), the third, seventh, and eighth terms on the right-hand side combine to yield
\begin{align}
\begin{split}
&  \frac{\phi^2 f^l E_l }{\sq} \nabla^i f_i
 - \frac{p\phi f^m E_m }{\sqsq}
  - \frac{p\phi^3 f^l E_l |\nabla f|^2_g }{\sqsq}\\
& =  \frac{\phi^2 f^l E_l }{\sq} \nabla^i f_i
- p\phi  f^m E_m
\\
& =  \frac{\phi^2 f^l E_l }{\sq} \nabla^i f_i
- \phi  f^m E_m g^{ij}
\left( \frac{\phi \nabla_{ij}f + \phi_i f_j + \phi_j f_i + \phi^2 \phi^l f_l f_i f_j}{\sq}  \right)
\\
& =  \frac{\phi^2 f^l E_l }{\sq} \nabla^i f_i
- \frac{\phi  f^m E_m}{\sq}
\left(\phi \nabla^i f_i + 2\phi^i f_i +  \phi^2 \phi^l f_l |\nabla f|_{g}^{2} \right)
\\
& =  - \frac{ 2\phi f^l E_l \phi^m f_m}{\sq} -  \frac{ \phi^3 f^l E_l \phi^m f_m |\nabla f|^2_g}{\sq} .
\end{split}
\label{big_3_7_8}
\end{align}
Write the second term on the right hand side of (\ref{big_identity}) as
\begin{gather}
\nabla^i \left( \frac{\phi^2 f^l E_l }{ \sq} \right) f_i =
 \frac{2 \phi  f^l E_l \phi^i f_i }{ \sq}  +
\phi^2 \nabla^i \left( \frac{f^l E_l }{ \sq} \right) f_i  .
\label{big_2}
\end{gather}
Insert (\ref{big_3_7_8}) and (\ref{big_2}) into (\ref{big_identity}).
The first term on the right-hand side of (\ref{big_3_7_8}) cancels with the
first term on the right-hand side of (\ref{big_2}), and
the second term on the right-hand side of (\ref{big_3_7_8}) cancels with the sixth term on the right-hand side of (\ref{big_identity}).
Therefore
\begin{align}
\begin{split}
 \overline{\dd} \, \overline{E} & =
\frac{1}{\sq} \dd E +
\phi^2 \nabla^i \left( \frac{f^l E_l }{ \sq} \right) f_i
 + g^{ij} \partial_j \left(\frac{1}{\sq}\right) E_i \\
& + \frac{\phi |\nabla f|_g^2 \phi^m E_m}{\sq}
- \frac{\phi^2 f^i f^j \nabla_j E_i }{(\sqsq)^\frac{3}{2}}
- \frac{\phi^2 f^i f^j}{\sqsq} \nabla_j \left( \frac{\phi^2 f^l E_l }{\sq} \right) f_i \\
&- \frac{ \phi^4 f^i f^j f^l E_l}{(\sqsq)^\frac{3}{2} } \nabla_j f_i
 -\frac{\phi^2 f^i f^j}{\sqsq} \partial_j \left( \frac{1}{\sq} \right) E_i
- \frac{\phi^3 f^i f^j \phi^m E_m f_i f_j }{ (\sqsq)^\frac{3}{2}  } \\
& - \frac{\phi^5 f^i f^j f^l E_l \phi^m f_m f_i f_j }{(\sqsq)^\frac{3}{2} }
+ \frac{\phi^3 f^i f^j p_{ij} f^m E_m }{(\sqsq)^2}
+ \frac{\phi^5 f^i f^j p_{ij} f^m E_m |\nabla f|^2_g }{(\sqsq)^2} .
\end{split}
\label{big_identity_2}
\end{align}

Notice that the last two terms combine into
\begin{equation*}
\frac{\phi^3 f^i f^j p_{ij} f^m E_m }{(\sqsq)^2}
+ \frac{\phi^5 f^i f^j p_{ij} f^m E_m |\nabla f|^2_g }{(\sqsq)^2}
= \frac{\phi^3 f^i f^j p_{ij} f^m E_m}{\sqsq} .
\end{equation*}
Moreover the third and fourth terms on the right hand side of (\ref{big_identity_2}) become
\begin{equation*}
-\frac{1}{2} \frac{\phi^2 g^{ij}   E_i  \partial_j|\nabla f|^2_g}{(\sqsq)^\frac{3}{2}}
+ \frac{\phi^3 |\nabla f|_g^4 \phi^m E_m }{(\sqsq)^\frac{3}{2}} .
\end{equation*}
Upon using this expression in (\ref{big_identity_2}), we find that the last term cancels with the
ninth term on the right-hand side of (\ref{big_identity_2}). Hence
\begin{align}
\begin{split}
 \overline{\dd} \, \overline{E} & =
\frac{1}{\sq} \dd E
+ \phi^2 f_i \nabla^i \left( \frac{f^l E_l }{ \sq} \right)
-\frac{1}{2} \frac{\phi^2 g^{ij}   E_i  \partial_j|\nabla f|^2_g }{(\sqsq)^\frac{3}{2}}
\\
&
- \frac{\phi^2 f^i f^j \nabla_j E_i }{(\sqsq)^\frac{3}{2}}
- \frac{\phi^2 |\nabla f|_g^2}{\sqsq} f_i\nabla^i \left( \frac{\phi^2 f^l E_l }{\sq} \right)
- \frac{ \phi^4 f^i f^j f^l E_l}{(\sqsq)^\frac{3}{2} } \nabla_j f_i \\
& -\frac{\phi^2 f^i f^j}{\sqsq} \partial_j \left( \frac{1}{\sq} \right) E_i
  - \frac{\phi^5 |\nabla f|_g^4 f^l E_l \phi^m f_m }{(\sqsq)^\frac{3}{2} }
+ \frac{\phi^3 f^i f^j p_{ij} f^m E_m}{\sqsq} .
\end{split}
\label{big_identity_3}
\end{align}
Now use (\ref{sff_graph}) to compute the last term in (\ref{big_identity_3})
\begin{align}
 \frac{\phi^3 f^i f^j p_{ij} f^m E_m}{\sqsq} =
\frac{\phi^4 f^l E_l f^i f^j \nabla_{ij}f }{(\sqsq)^\frac{3}{2}}
+ \frac{2 \phi^3 f^l E_l |\nabla f|_g^2 \phi^m f_m  }{(\sqsq)^\frac{3}{2}}
+ \frac{ \phi^5 f^l E_l |\nabla f|_g^4 \phi^m f_m  }{(\sqsq)^\frac{3}{2}},
\label{f_i_f_j_h_ij_term}
\end{align}
and observe that the first term on the right-hand side of (\ref{f_i_f_j_h_ij_term}) cancels the sixth term on the
right-hand side of (\ref{big_identity_3}), while the last term in (\ref{f_i_f_j_h_ij_term})
cancels the next to last term in (\ref{big_identity_3}). Therefore
\begin{align}
\begin{split}
 \overline{\dd} \, \overline{E} & =
\frac{1}{\sq} \dd E
+ \phi^2 f_i \nabla^i \left( \frac{f^l E_l }{ \sq} \right)
-\frac{1}{2} \frac{\phi^2 g^{ij}   E_i  \partial_j|\nabla f|^2_g }{(\sqsq)^\frac{3}{2}} \\
& - \frac{\phi^2 f^i f^j \nabla_j E_i }{(\sqsq)^\frac{3}{2}}
 - \frac{\phi^2 |\nabla f|_g^2}{\sqsq} f_i\nabla^i \left( \frac{\phi^2 f^l E_l }{\sq} \right)  \\
 & -\frac{\phi^2 f^i f^j}{\sqsq} \partial_j \left( \frac{1}{\sq} \right) E_i
   + \frac{2 \phi^3  f^l E_l \phi^m f_m |\nabla f|^2_g}{(\sqsq)^\frac{3}{2}} .
\end{split}
\label{big_identity_4}
\end{align}

Next, consider the second, fourth, and fifth terms on the right-hand side of (\ref{big_identity_4}). They give
\begin{align}
\begin{split}
& \phi^2 f_i \nabla^i \left( \frac{f^l E_l }{ \sq} \right)
- \frac{\phi^2 f^i f^j \nabla_j E_i }{(\sqsq)^\frac{3}{2}}
- \frac{\phi^2 |\nabla f|_g^2}{\sqsq} f_i\nabla^i \left( \frac{\phi^2 f^l E_l }{\sq} \right)   \\
& =
\frac{\phi^2 f^l f_i \nabla^i E_l }{\sq }
+ \phi^2 f_i \nabla^i \left(\frac{f^l}{\sq}\right) E_l -
\frac{\phi^2 f^i f^j \nabla_j E_i }{(\sqsq)^\frac{3}{2} } \\
& - \frac{\phi^2 |\nabla f|_g^2 }{\sqsq} \frac{\phi^2 f^l}{\sq} (\nabla^i E_l ) f_i
- \frac{\phi^2 |\nabla f|_g^2 }{\sqsq} \nabla^i \left( \frac{\phi^2 f^l }{\sq} \right) E_l f_i .
\end{split}
\label{2_4_5}
\end{align}
The first, third, and fourth terms on the right-hand side cancel
\begin{align}
\begin{split}
& \frac{\phi^2 f^l f_i \nabla^i E_l }{\sq }
- \frac{\phi^2 f^i f_j \nabla^j E_i }{(\sqsq)^\frac{3}{2} }
- \frac{\phi^2 |\nabla f|_g^2 }{\sqsq} \frac{\phi^2 f^l}{\sq} (\nabla^i E_l ) f_i  \\
& =
\frac{ \phi^2 f^l f_i \nabla^i E_l }{\sq} \left( 1 - \frac{1}{\sqsq} - \frac{\phi^2 |\nabla f|_g^2 }{\sqsq} \right)
= 0,
\end{split}
\label{equal_zero}
\end{align}
whereas the second and fifth terms of (\ref{2_4_5}) become
\begin{align}
\begin{split}
& \phi^2 f_i \nabla^i \left(\frac{f^l}{\sq}\right) E_l
- \frac{\phi^2 |\nabla f|_g^2 }{\sqsq} \nabla^i \left( \frac{\phi^2 f^l }{\sq} \right) E_l f_i \\
  &  = \phi^2 f_i \nabla^i \left(\frac{f^l}{\sq}\right) E_l
- \frac{ \phi^4 |\nabla f|_g^2}{\sqsq} f_i \nabla^i \left(\frac{ f^l}{\sq} \right) E_l
 - \frac{\phi^2 |\nabla f|_g^2 f^l E_l f_i \nabla^i \phi^2 }{(\sqsq) \sq } \\
& = \frac{ \phi^2 }{\sqsq} f_i \nabla^i \left( \frac{f^l }{\sq}\right) E_l
- \frac{2 \phi^3 f^l E_l \phi^i f_i |\nabla f|_g^2}{(\sqsq)^\frac{3}{2}  }.
\end{split}
\label{other_two}
\end{align}
Plugging (\ref{2_4_5}), (\ref{equal_zero}), and (\ref{other_two}) into (\ref{big_identity_4}) produces
\begin{align}
\begin{split}
 \overline{\dd} \, \overline{E} & =
\frac{1}{\sq} \dd E
+ \frac{ \phi^2 }{\sqsq} f_i \nabla^i \left( \frac{f^l }{\sq}\right) E_l
-\frac{1}{2} \frac{\phi^2 g^{ij}   E_i  \partial_j|\nabla f|^2_g }{(\sqsq)^\frac{3}{2}} \\
&  -\frac{\phi^2 f^i f^j}{\sqsq} \partial_j \left( \frac{1}{\sq} \right) E_i ,
\nonumber
\end{split}
%\label{big_identity_5}
\end{align}
after canceling the last term in (\ref{other_two}) with the last term in (\ref{big_identity_4}).

Lastly, expand the derivative in the second term on the right to find
\begin{align}
\begin{split}
 \overline{\dd} \, \overline{E}
& =
\frac{1}{\sq} \dd E
+ \frac{\phi^2 f_i }{\sqsq} \left( f^l \nabla^i \left( \frac{1}{\sq} \right) + \frac{1}{\sq} \nabla^i f^l \right )E_l \\
& -\frac{1}{2} \frac{\phi^2 E_i  \nabla^i(|\nabla f|^2_g )}{(\sqsq)^\frac{3}{2}}
 -\frac{\phi^2 f^l E_l }{\sqsq} f^j \nabla_j \left( \frac{1}{\sq} \right) \\
& =
\frac{1}{\sq} \dd E
+ \frac{ \phi^2 E_l f_i \nabla^i f^l }{(\sqsq)^\frac{3}{2} }
 -\frac{1}{2} \frac{\phi^2 E_i  \nabla^i|\nabla f|^2_g }{(\sqsq)^\frac{3}{2}} .
\end{split}
\nonumber
\end{align}
The last two terms cancel, yielding the desired result.
\end{proof}

\begin{lemma}
If $\phi$ remains bounded and $f(x)\rightarrow 0$ as $|x|\rightarrow\infty$, then $e=\overline{e}$.
\end{lemma}

\begin{proof}
This follows directly from the definition of total charge and the formula \eqref{E_bar_Riemmanian}.
\end{proof}

\section{Lorentzian case\label{Lorentzian_case}}

In this section we will establish \eqref{hat_E_i_equal_E_i}. As in the previous section, the validity of the results in this appendix is independent
of any equation that $f$ may satisfy.

Consider the Lorentzian static spacetime $(\Si \times \RR, \, \overline{g} - \phi^2 \, dt^2)$ with electromagnetic field strength
\begin{gather}
F = \left(
 \begin{array}{cccc}
  0 & \phi \overline{E}_1 & \phi \overline{E}_2 & \phi \overline{E}_3 \\
- \phi \overline{E}_1 & 0 & 0 & 0 \\
- \phi \overline{E}_2 & 0 & 0 & 0 \\
- \phi \overline{E}_3 & 0 & 0 & 0
 \end{array}
\right).
\nonumber
\end{gather}
As usual $\overline{g}=g+\phi^{2}df^{2}$, and $\overline{E}$ is the electric field on $\Sigma$ given by \eqref{E_bar_Riemmanian}. We will regard $\mathcal{M}=\{t=f(x)\}$ as
a graph inside this spacetime, since the induced metric on the graph is $g$. Note that in Section \ref{case_k_not_zero} this surface was denoted by
$I(\mathcal{M})$. Denote by $\widehat{E}$ the electric field induced on $\cM$, as explained at the beginning of the appendix.

\begin{lemma}
If $E$ is the electric field found in the definition of $\overline{E}$, then
\begin{equation*}\label{40}
E=\widehat{E}.
\end{equation*}
\end{lemma}

\begin{proof}
The unit normal to $\cM$ is given by
\begin{equation*}
N = \frac{ \phi^{-1} \partial_t  + \phi f^{\bar{l}} \partial_l }{\sqrt{ 1 - \phi^2 |\nabla f|_{\overline{g}}^2} },
\end{equation*}
where a barred index indicates that it is raised with respect to the $\overline{g}$ metric, that is
$f^{\bar{l}} = \overline{g}^{lj} f_j$.
Compute
\begin{align}
\begin{split}
 \widehat{E}_i & = F(N , X_i) =
F\left( \frac{ \phi^{-1} \partial_t  + \phi f^{\bar{l}} \partial_l }{\sqrt{ 1 - \phi^2 |\nabla f|_{\overline{g}}^2} } ,
\partial_i + f_i \partial_t \right) \\
& = \frac{1}{ \sqrt{ 1 - \phi^2 |\nabla f|_{\overline{g}}^2} }
\left( \phi^{-1} F(\partial_t, \partial_i) + \phi^{-1} f_i F(\partial_t, \partial_t )
+ \phi f^{\bar{l}} F(\partial_l, \partial_i) + \phi f^{\bar{l}} f_i F(\partial_l, \partial_t) \right),
\end{split}
\nonumber
\end{align}
and use $F(\partial_t, \partial_i) = \phi \overline{E}_i$,
$F(\partial_t, \partial_t) = F(\partial_l, \partial_i) = 0$, and
$F(\partial_l, \partial_t) = - \phi \overline{E}_i$ to show that
\begin{align}
\begin{split}
 \widehat{E}_i
& = \frac{1}{ \sqrt{ 1 - \phi^2 |\nabla f|_{\overline{g}}^2} }
\left( \overline{E}_i - \phi^2 f_i f^{\bar{l}}\,  \overline{E}_l  \right).
\end{split}
\label{hat_E_i}
\end{align}
The goal is to express everything in terms of unbarred quantities.
Observe that
\begin{align}
 \begin{split}
f^{\bar{l}}\,  \overline{E}_l
& = \overline{g}^{lj} f_l \overline{E}_j   \\
& = \left( g^{lj} - \frac{\phi^2 f^l f^j }{1+ \phi^2 |\nabla f|_g^2 } \right) f_l
\left( \frac{E_{j}+\phi^{2} f_j f^k E_{k}}{ \sqrt{1+\phi^{2}|\nabla f|_g^{2}} } \right) \\
& =
\frac{1}{ \sqrt{1+\phi^{2}|\nabla f|_g^{2}} } f^l E_l .
 \end{split}
\label{f_E_bar_g_bar}
\end{align}
Now input (\ref{f_E_bar_g_bar}) and (\ref{E_bar_Riemmanian}) into
(\ref{hat_E_i}) to find
\begin{align}
\begin{split}
 \widehat{E}_i
& = \frac{1}{ \sqrt{ 1 - \phi^2 |\nabla f|_{\overline{g}}^2 } }
\left(
\frac{E_{i}+\phi^{2} f_i  f^l E_{l} }{ \sqrt{1+\phi^{2}|\nabla f|_g}^{2} }
- \phi^2 f_i \frac{1}{ \sqrt{1+\phi^{2}|\nabla f|_g^{2}} } f^l E_l     \right) \\
& = \frac{E_i}{ \sqrt{ ( 1 - \phi^2 |\nabla f|_{\overline{g}}^2 ) (1+\phi^{2}|\nabla f|_g^{2}) } } .
\nonumber
\end{split}
\end{align}
Lastly, note that
\begin{align}
\begin{split}
 ( 1 - \phi^2 |\nabla f|_{\overline{g}}^2 ) (1+\phi^{2}|\nabla f|_g^{2}) & =
 \left(1-\phi^{2}\left(g^{ij}-\frac{\phi^{2}f^{i}f^{j}}{1+\phi^{2}|\nabla f|_{g}^{2}}\right)f_{i}f_{j}\right)(1+\phi^{2}|\nabla f|_g^{2}) \\
 &=\left(1-\frac{\phi^{2}|\nabla f|_{g}^{2}}{1+\phi^{2}|\nabla f|_{g}^{2}}\right)(1+\phi^{2}|\nabla f|_g^{2}) \\
 &=1.
\nonumber
\end{split}
\end{align}

\end{proof}


\begin{thebibliography}{99}
\bibitem{ADM} R. Arnowitt, S. Deser, and  C. Misner, \emph{Coordinate invariance and energy
expressions in General Relativity}, Phys. Rev., \textbf{122} (1961), 997-1006.
\bibitem{BT} R. Bassanezi, and I. Tamanini, \emph{Subsolutions to the least area problem and the
`minimal hull' of a bounded set $\mathbb{R}^{n}$}, Ann. Univ. Ferrara Sez. VII (N.S.), \textbf{30}
(1984), 27–40.
\bibitem{Bray1} H. Bray, \emph{Proof of the Riemannian Penrose inequality using the positive mass
theorem}, J. Differential Geom., \textbf{59} (2001), 177-267.
\bibitem{BK1} H. Bray, and M. Khuri, \emph{P.D.E.'s which imply the Penrose conjecture}.
Asian J. Math., \textbf{15} (2011), no. 4, 557-610. arXiv:0905.2622v1
\bibitem{BK2} H. Bray, and M. Khuri, \emph{A Jang equation approach to the Penrose inequality},
Discrete and Continuous Dynamical Systems A, \textbf{27} (2010), no. 2, 741–766. arXiv:0910.4785v1
\bibitem{C} Y. Choquet-Bruhat, \emph{General Relativity and the Einstein Equations}, Oxford University Press, 2009.
\bibitem{CGD} P. Chrusciel, G. Galloway, and D. Pollack, \emph{Mathematical General Relativity: a sampler},
Bull. Amer. Math. Soc. (N.S.), \textbf{47} (2010), no. 4, 567-638. arXiv: 1004.1016v2
%\bibitem{DJR} S. Dain, J. Jaramillo, and M. Reiris, \emph{Area-charge inequality for black holes},
%Classical Quantum Gravity, \textbf{29} (2012), no. 3, 035013. arXiv: 1109.5602v1
\bibitem{Da} T. Damour, General relativity today. \emph{Gravitation and experiment}, 1–49, Prog. Math. Phys., \textbf{52}, Birkhäuser, Basel, 2007. arXiv: 0704.0754v1
\bibitem{GH} R. Geroch and G. Horowitz, in Einstein Centenary Volume, edited by S. Hawking and W. Israel (Plenum, NY, 1979).
\bibitem{Hay} S. Hayward, \emph{Inequalities relating area, energy, surface gravity and charge of black holes}, Phys. Rev. Lett.,
\textbf{81} (1998), no. 21, 4557–4559. arXiv: 9807003v1
\bibitem{HE} S. Hawking, and G. Ellis, \emph{The Large Structure of Space-Time}, Cambridge Monographs on Mathematical
Physics, Cambridge University Press, 1973.
\bibitem{HI} G. Huisken, and T. Ilmanen, \emph{The inverse mean curvature flow and the Riemannian
Penrose inequality}, J. Differential Geom., \textbf{59} (2001), 353-437.
\bibitem{HK} Q. Han, and M. Khuri, \emph{Existence and blow up behavior for solutions
of the generalized Jang equation}, Comm. Partial Differential Equations, \textbf{38} (2013), 2199-2237. arXiv:1206.0079 
\bibitem{Ja} P.-S. Jang, \emph{On the positivity of energy in General Relaitivity}, J. Math. Phys., \textbf{19}
(1978), 1152-1155.
\bibitem{Ja2} P.-S. Jang, \emph{Note on cosmic censorship}, Phys. Rev. D, \textbf{20} (1979), no. 4, 834–838.
\bibitem{KW} M. Khuri, and G. Weinstein, \emph{Rigidity in the positive mass theorem with charge}, J. Math. Phys., \textbf{54} (2013), 092501. arXiv:1307.5499
\bibitem{Mar} M. Mars, \emph{Present status of the Penrose inequality}, Classical Quantum Gravity, \textbf{26} (2009), no. 19,
193001. arXiv: 0906.5566v1
\bibitem{Pen} R. Penrose, \emph{Naked singularities}, Ann. New York Acad. Sci., \textbf{224} (1973), 125-134.
\bibitem{SY2} R. Schoen, and S.-T. Yau, \emph{On the proof of the positive mass
conjecture in General Relativity}, Comm. Math. Phys., \textbf{65} (1979), no. 1, 45-76.
\bibitem{SY3} R. Schoen, and S.-T. Yau, \emph{Proof of the positive mass theorem II}, Comm. Math. Phys.,
\textbf{79} (1981), 231-260.
\bibitem{WY} G. Weinstein, S. Yamada, \emph{On a Penrose inequality with charge}, Comm. Math. Phys., \textbf{257} (2005), no. 3, 703–723.
\bibitem{W} E. Witten, \emph{A new proof of the positive energy theorem}, Comm. Math. Phys., \textbf{80} (1981) 381-402.
\end{thebibliography}
\end{document}